\documentclass{amsart}

\usepackage{amssymb,amsfonts}
\usepackage[all,arc]{xy}
\usepackage{enumerate}
\usepackage{mathrsfs}
\usepackage{tikz}
\usetikzlibrary{matrix,arrows}

\newtheorem{thm}{Theorem}[section]
\newtheorem{cor}[thm]{Corollary}
\newtheorem{prop}[thm]{Proposition}
\newtheorem{lem}[thm]{Lemma}

\theoremstyle{definition}
\newtheorem{defn}[thm]{Definition}

\theoremstyle{remark}

\begin{document}

\title{Sums-of-Squares Formulas over Algebraically Closed Fields}
\author{Melissa Lynn}

\begin{abstract}
In this paper, we consider whether existence of a sums-of-squares formula depends on the base field. We reformulate the question of existence as a question in algebraic geometry. We show that, for large enough $p$, existence of sums-of-squares formulas over algebraically closed fields is independent of the characteristic. We make the bound on $p$ explicit, and we prove that the existence of a sums-of-squares formula of fixed type over an algebraically closed field is theoretically (though not practically) computable.
\end{abstract}

\maketitle

\section{Introduction}

Let $F$ be a field of characteristic not $2$. A \emph{sums-of-squares formula of type $[r,s,n]$ over $F$} is an identity of the form
\[
(x_1^2+\cdots + x_r^2)(y_1^2+\cdots +y_s^2)=z_1^2+\cdots +z_n^2
\]
where each $z_i$ is a bilinear expression in the $x$'s and $y$'s over $F$.

The existence of sums-of-squares formulas has been extensively studied since they arose in relation to real normed division algebras, and Hurwitz \cite{Hurwitz1} \cite{Hurwitz2} posed the general question: for what $r,s,n$ does a sums-of-squares formula of type $[r,s,n]$ exist over a field $F$?

Sums-of-squares formulas were used to prove Hurwitz's theorem that the only real normed division algebras are the real numbers, complex numbers, quaternions, and octonions. They continue to be of interest for their relationship to problems in topology and geometry: they provide immersions of projective space into Euclidean space, and they provide Hopf maps. Over arbitrary fields, sums-of-squares formulas provide examples of compositions of quadratic forms.

A detailed treatment of the historical development of the subject and past results can be found in Shapiro's book \cite{ShapiroBook}.

Whether existence of a sums-of-squares formula depends on the base field remains an open question. Existence over arbitrary fields is particularly interesting, because formulas over finite fields can be found using computational methods, and these formulas could then yield formulas in the classical characteristic 0 setting.

Most results on the existence of sums-of-squares formulas have been obtained only for fields of characteristic $0$. For some very special cases of $r,s,n$, Adem \cite{Adem1} \cite{Adem2} and Yuzvinsky \cite{Yuzvinsky} have settled the question of existence over arbitrary fields. For fixed $r$ and $s$, topological and geometric considerations have provided lower bounds on the smallest $n$ for which a sums-of-squares formula of type $[r,s,n]$ can exist over $\mathbb{R}$. Recently, Dugger and Isaksen \cite{Dugger1} \cite{Dugger2} \cite{Dugger3} show these lower bounds are valid over any field $F$, improving on the results of Shapiro and Szyjewski \cite{SS}. Xie \cite{Xie} also contributed in this vein.

In this paper, we provide a new formulation of the question of existence of sums-of-squares formulas as a question in algebraic geometry, defining the variety of sums-of-squares formulas. We use this reformulation to prove that in the question of existence for any $r,s,n$ and large enough $p$, existence over algebraically closed fields of characteristic $0$ and $p$ are equivalent. We also prove that existence of sums-of-squares formulas of type $[r,s,n]$ is theoretically computable, although not practically computable.

We begin by showing:

\begin{thm}
If $F\subset K$ are algebraically closed fields, there is a sums-of-squares formula of type $[r,s,n]$ over $F$ if and only if there is one over $K$.
\end{thm}

This theorem means that, in determining existence of formulas over algebraically closed fields, it suffices to consider the fields $\bar{\mathbb{Q}}$ and $\bar{\mathbb{F}}_p$ for $p$ prime.

Next, we prove that for all but possibly finitely many $p$, existence in characteristic $0$ and $p$ are equivalent.

\begin{thm}
For algebraically closed fields, existence of sums-of-squares formulas of type $[r,s,n]$ over fields of characteristic $0$ and $p$ are equivalent for all but finitely many $p$.
\end{thm}

As a corollary of this theorem, we show that existence of sums-of-squares formulas can always be detected over finite fields:

\begin{cor}
If a sums-of-squares formula of type $[r,s,n]$ exists over any field $F$, then a sums-of-squares formula of type $[r,s,n]$ exists over some finite field.
\end{cor}

We then improve upon this result by providing an explicit degree bound on this finite field.

\begin{thm}
Fix $[r,s,n]$ and prime $p$. If there is a sums-of-squares formula of type $[r,s,n]$ over $\bar{\mathbb{F}}_p$, then there is a sums-of-squares formula of type $[r,s,n]$ over $\mathbb{F}_{p^k}$ for some $k\leq 2\cdot 17^{rsn+r^2s^2}$.
\end{thm}

This implies that existence of a sums-of-squares formula over an algebraically closed field of finite characteristic is computable, since determining existence involves checking finitely many finite fields.

Finally, we consider how large $p$ must be to guarantee the characteristic $0$ and characteristic $p$ cases are equivalent. We provide an explicit bound for this:

\begin{thm}
Fix $r,s,n$. Consider a prime $p$ with
\[\label{*}
p>\left(2^{3\cdot 2^q-2}\right)^{\left(\left(5\cdot 2^q-3\right)^{m+1}-1\right)/\left(5\cdot 2^q-4\right)}
\]
where $q=\left(\begin{array}{c}
rsn+4^{2^{n-1}+1}\\
4^{2^{n-1}+1}\end{array}\right)$ and $m=\left(\begin{array}{c}
rsn+2\cdot 4^{2^{n-1}}\\
2\cdot 4^{2^{n-1}}
\end{array}\right)$.

Then a sums-of-squares formula of type $[r,s,n]$ exists over an algebraically closed field of characteristic $0$ if and only if a sums-of-squares formula of type $[r,s,n]$ exists over an algebraically closed field of characteristic $p$.
\end{thm}

Combining this with the previous result, we have shown that existence of a sums-of-squares formula over an algebraically closed field is theoretically computable.

\begin{thm}
Existence of a sums-of-squares formula of type $[r,s,n]$ over any algebraically closed field is computable.
\end{thm}

Unfortunately, our bounds are too large for this to be practical to compute.

\section{Summary of Methods}

We begin by reformulating the question of existence of sums-of-squares formulas as one in algebraic geometry. This is done by observing that a sums-of-squares formula is given by the coefficients of $x_jy_k$ in each $z_i$, which must satisfy certain polynomial equations. Thus, giving a sums-of-squares formula is equivalent to giving a solution to this set of polynomial equations, and a sums-of-squares formula exists over an algebraically closed field if and only if the corresponding variety is non-empty.

Denote by $A_{rsn}^F$ the coordinate ring corresponding to this set of polynomials. Let $X_{rsn}^F=\textrm{Spec }A_{rsn}^F$ be the corresponding variety, which we call the \emph{variety of sums-of-squares formulas of type $[r,s,n]$ over $F$}. Note that $A_{rsn}^F$ and $X_{rsn}^F$ are defined over the integers.

As a first step towards the main theorem we observe that if $F\subset K$ are algebraically closed fields, there is a sums-of-squares formula of type $[r,s,n]$ over $F$ if and only if there is one over $K$. This is a consequence of the Hilbert Nullstellensatz.

Using a reduction argument, we also show that if there is no sums-of-squares formula of type $[r,s,n]$ over $\bar{\mathbb{Q}}$, then there is no sums-of-squares formula of type  $[r,s,n]$ over $\bar{\mathbb{F}}_p$ for all but finitely $p$.

Working in the other direction, we take $\alpha_{ijk}$ to be the coefficients of a sums-of-squares formula over $\bar{\mathbb{Q}}$. We consider the coefficients of the minimal polynomials of the $\alpha_{ijk}$ to show that if there is a sums-of-squares formula of type $[r,s,n]$ exists over $\bar{\mathbb{Q}}$, then a sums-of-squares formula of the same type exists over $\bar{\mathbb{F}}_p$ for all but finitely many primes $p$.

Thus, we have that, for algebraically closed fields, existence of a sums-of-squares formula over fields of characteristic $0$ and characteristic $p$ are equivalent for all but finitely many $p$.

As a corollary to this theorem, we show that if a sums-of-squares formula exists over any field, then a sums-of-squares formula exists over some finite field $\mathbb{F}_{p^m}$. This follows from the main theorem by noting that a formula consists of only finitely many coefficients. In the second section, we improve on this result by providing an upper bound for how large such a finite field must be. This is accomplished by considering the zeta function and applying a theorem due to Bombieri \cite{Bombieri}.

Next, we consider how large $p$ must be to ensure equivalence between the characteristic $0$ and characteristic $p$ cases.

We have reduced the question of existence of sums-of-squares formulas to the question of whether or not a certain ideal is proper. This can be decided using Gr{\"o}bner bases, although not efficiently enough to be useful in this case.

However, Gr{\"o}bner bases can be computing using Buchberger's algorithm, and by analyzing the coefficients which can appear at each step in Buchberger's algorithm, we are able to prove equivalence between the characteristic $0$ and $p$ cases for ``large enough'' $p$.

By analyzing coefficients, we are able to give an explicit bound for $p$ for which this is true.

Finally, we combine this with the previous result providing a degree bound to show that existence of a sums-of-squares formula of type $[r,s,n]$ over an algebraically closed field is theoretically computable.

\section{Reformulation}

We would like to consider the question of when sums-of-squares formulas exist, i.e. when, for a fixed field $F$ with $\textrm{char}(F)\neq 2$, we have
\[
	(x_1 ^2 + \cdots + x_r ^2)(y_1 ^2 + \cdots + y_s ^2) 
			= z_1 ^2 + \cdots + z_n ^2
\]
where the $z_i$ are bilinear expressions in the $x$'s and $y$'s with coefficients in $F$.

We reformulate this question as one about an ideal in a polynomial ring.

We can write
\[
	z_i = \sum_{j,\; k}\alpha_{ijk}x_jy_k
\]
Then, expanding the left side of the formula, we have:
\[
	\sum_{j,\; k}x_j^2y_k^2
\]
Expanding the right side of the formula, we get:
\[
\begin{array}{rcl}
	\sum_{i,j,k} (\alpha_{ijk}^2x_j^2y_k^2) & + & \sum_{i,j<j',k}(2\alpha_{ijk}\alpha_{ij'k}x_jx_{j'}y_k^2)\\
	& + & \sum_{i,j,k<k'} (2\alpha_{ijk}\alpha_{ijk'}x_j^2y_ky_{k'})\\
	& + & \sum_{i,j<j',k<k'} (2(\alpha_{ijk}\alpha_{ij'k'}+\alpha_{ijk'}\alpha_{ij'k})x_jx_{j'}y_ky_{k'})
\end{array}
\]
Grouping terms, we see that the coefficient of $x_j^2 y_k^2$ on the right side is:
\[
	\sum_i\alpha_{ijk}^2
\]
the coefficient of $x_jx_{j'}y_k^2$ where $j<j'$ is:
\[
	\sum_i 2\alpha_{ijk}\alpha_{ij'k}
\]
the coefficient of $x_j^2y_ky_{k'}$ where $k<k'$ is:
\[
	\sum_i 2\alpha_{ijk}\alpha_{ijk'}
\]
and the coefficient of $x_jx_{j'}y_ky_{k'}$ where $j<j'$ and $k<k'$ is:
\[
	\sum_i 2(\alpha_{ijk}\alpha_{ij'k'}+\alpha_{ijk'}\alpha_{ij'k})
\]

Comparing these expressions with the coefficients on the left side of the formula, we see that existence of a sums-of-squares formula over $F$ is equivalent to the existence of $\alpha_{ijk}\in F$ satisfying
\[
\left\{\begin{array}{ll}
	\sum_i\alpha_{ijk}^2 = 1 & \textrm{for all } j,k\\
	\sum_i \alpha_{ijk}\alpha_{ij'k} = 0 & \textrm{for } j < j' \textrm{ and all }k\\
	\sum_i \alpha_{ijk}\alpha_{ijk'} = 0 & \textrm{for all } j \textrm{ and } k < k'\\
	\sum_i (\alpha_{ijk}\alpha_{ij'k'}+\alpha_{ijk'}\alpha_{ij'k}) = 0 & \textrm{for } j<j' \textrm{ and } k<k'\\
\end{array}\right\}
\]
Note that we can drop the coefficient 2, since $\textrm{char}(F)\neq 2$.\\
Changing notation, we are asking if the set of polynomials
\[
\left\{\begin{array}{ll}
	\sum_i (x_{ijk}^2) -1 & \textrm{for all } j,k\\
	\sum_i x_{ijk}x_{ij'k} & \textrm{for } j < j' \textrm{ and all }k\\
	\sum_i x_{ijk}x_{ijk'} & \textrm{for all } j \textrm{ and } k < k'\\
	\sum_i (x_{ijk}x_{ij'k'}+x_{ijk'}x_{ij'k}) & \textrm{for } j<j' \textrm{ and } k<k'\\
\end{array}\right\}
\]
has a zero.\\
Let $I$ be the ideal generated by this set of polynomials in $F[\{x_{ijk}\}]$.
The set of polynomials has a zero in an algebraic closure of $F$ if and only if $I$ is a proper ideal.
Thus, we have the following reformulation for the question of existence of sums-of-squares formulas:

\begin{prop}
Let $I$ be the ideal generated by the following set of polynomials in $S=F[\{x_{ijk}\}]$ (where $1\leq i \leq n$, $1\leq j\leq r$, $1\leq k\leq s$):
\[
\left\{\begin{array}{ll}
	\sum_i (x_{ijk}^2) -1 & \textrm{for all } j,k\\
	\sum_i x_{ijk}x_{ij'k} & \textrm{for } j < j' \textrm{ and all }k\\
	\sum_i x_{ijk}x_{ijk'} & \textrm{for all } j \textrm{ and } k < k'\\
	\sum_i (x_{ijk}x_{ij'k'}+x_{ijk'}x_{ij'k}) & \textrm{for } j<j' \textrm{ and } k<k'\\
\end{array}\right\}
\]
A sums-of-squares formula of type $[r,s,n]$ exists over an algebraic closure of $F$ if and only if $I\subset S$ is a proper ideal.
\end{prop}

We set the notation $A_{rsn}^F=S/I$, where $S$ and $I$ are as above. Let $X_{rsn}^F$ denote the variety of sums-of-squares formulas of type $[r,s,n]$ over $F$. We define $A_{rsn}^\mathbb{Z}$ and $X_{rsn}^\mathbb{Z}$ similarly.

In particular, in the algebraically closed case, this reformulation can be stated as follows:

\begin{prop}
Let $F$ be an algebraically closed field. Then a sums-of-squares formula of type $[r,s,n]$ exists over $F$ if and only if $A_{rsn}^F\neq 0$.
\end{prop}

By the Hilbert Nullstellensatz, this is also equivalent to $X_{rsn}^F$ being a nonempty scheme over $F$.

\section{Preliminary Results}

We begin with an immediate consequence of our reformulation and the Hilbert Nullstellensatz, which allows us to consider only $\bar{\mathbb{Q}}$ and $\bar{\mathbb{F}}_p$ for $p$ prime.

\begin{thm}
If $F\subset K$ are algebraically closed fields, there is a sums-of-squares formula of type $[r,s,n]$ over $F$ if and only if there is one over $K$.

In particular, in determining existence of formulas over algebraically closed fields, it suffices to consider the fields $\bar{\mathbb{Q}}$ and $\bar{\mathbb{F}}_p$ for $p$ prime.
\end{thm}

Consequently,

\begin{cor}
A sums of squares formula exists over $\mathbb{C}$ if and only if it exists over $\bar{\mathbb{Q}}$.
\end{cor}

We now make a reduction argument about what happens if there are no formulas of type $[r,s,n]$ over $\bar{\mathbb{Q}}$ (or, equivalently, any algebraically closed field of characteristic $0$):

\begin{prop} \label{Finp}
If there is no sums-of-squares formula of type $[r,s,n]$ over $\bar{\mathbb{Q}}$, then there is no sums-of-squares formula of type $[r,s,n]$ over $\bar{\mathbb{F}}_p$ for all but finitely many $p$.
\end{prop}

\begin{proof}
Suppose there is no sums-of-squares formula of type $[r,s,n]$ over $\bar{\mathbb{Q}}$. Then
\[
A_{rsn}^\mathbb{Q}=A_{rsn}^\mathbb{Z}\otimes_{\mathbb{Z}}\mathbb{Q}=0
\]
so for all $f\in A_{rsn}^\mathbb{Z}$, we must have $af = 0$ for some $0\neq a\in \mathbb{Z}$. In particular, there is $a\in \mathbb{Z}$ such that $a\cdot 1 = 0$. Then $af=0$ for all $f$, so we can find one $a$ that works for all $f$.

Now, let $p$ be a prime not dividing $a$. Then $a$ is invertible in $\mathbb{Z}/p\mathbb{Z}$, so
\[
A_{rsn}^\mathbb{Z}\otimes_\mathbb{Z}\mathbb{Z}/p\mathbb{Z}=0
\]
Thus, for all primes $p$ not dividing $a$, there is no sums-of-squares formula of type $[r,s,n]$ over $\bar{\mathbb{F}}_p$.

There are only finitely many primes dividing $a$, so this completes the proof.
\end{proof}

Now, we will prove that the characteristic $0$ and $p$ cases coincide for all but finitely many $p$.

\begin{prop} 
If there is a sums-of-squares formula of type $[r,s,n]$ over any field of characteristic $0$, then there is a sums-of-squares formula of type $[r,s,n]$ over every algebraically closed field of characteristic $p$ for all but finitely many $p$.
\end{prop}

Note that if a sums-of-squares formula exists over a field $F$ of characteristic $0$, then that same formula is valid over the algebraic closure $\bar{F}$. Then, by our previous theorem, there is a sums-of-squares formula over $\bar{\mathbb{Q}}$.

\begin{proof}
Suppose there is a sums-of-squares formula of type $[r,s,n]$ over $\bar{\mathbb{Q}}$ with coefficients $\alpha_{ijk}$. Let $K$ be the number field generated by the $\alpha_{ijk}$, and $O_K$ the ring of algebraic integers in $K$. Since there are finitely many $\alpha_{ijk}$, there are finitely many coefficients in their minimal polynomials. Let $S$ be the set of primes that are divisors of the denominators of the coefficients of these minimal polynomials. Then, for prime $p$ not in $S$, for every prime ideal $P$ in $O_K$ lying over $(p)$, the coefficients $\alpha_{ijk}$ are integral at $P$. Reducing modulo $P$, we obtain a sums-of-squares formula over the field $O_K/P$, which is a finite field extension of $\mathbb{F}_p$. Hence there is a sums-of-squares formula of type $[r,s,n]$ over $\bar{\mathbb{F}}_p$ for $p$ not in the finite set $S$.
\end{proof}

Putting these two propositions together, we have shown:

\begin{thm} \label{Ch4Thm}
For algebraically closed fields, existence of sums-of-squares formulas of type $[r,s,n]$ over fields of characteristic $0$ and $p$ are equivalent for all but finitely many $p$.
\end{thm}

Of course, we would like to be able to eliminate the condition ``for all but finitely many $p$.''

Considering the scheme $X_{rsn}^\mathbb{Z}$ over $\mbox{Spec }\mathbb{Z}$, we have shown that the generic fiber is nonempty if and only if ``almost all'' fibers are nonempty. If we could somehow show that this scheme is flat, then since it would have open image, if the fiber over a prime $p$ were nonempty, then all fibers would be nonempty. Showing in addition that this image is closed would establish the independence of existence of a sums-of-squares formula from the base field in the algebraically closed case. However, showing that this scheme is flat seems to be very difficult.

We also have the following corollary:

\begin{cor}
If a sums-of-squares formula of type $[r,s,n]$ exists over a field $F$, then a sums-of-squares formula of type $[r,s,n]$ exists over some finite field.
\end{cor}

\begin{proof}
If $F$ has characteristic $0$, then there is a sums-of-squares formula of type $[r,s,n]$ over $\bar{\mathbb{F}}_p$ for some $p$.

If $F$ has characteristic $q$, then there is a sums-of-squares formula of type $[r,s,n]$ over $\bar{F}$, hence over $\bar{\mathbb{F}}_q$.

In either case, we have a sums-of-squares formula of type $[r,s,n]$ over some $\bar{\mathbb{F}}_q$. This sums-of-squares formula is given by finitely many coefficients $\alpha_{ijk}$, so we have a sums-of-squares formula over $\mathbb{F}_q[\{\alpha_{ijk}\}]$. The $\alpha_{ijk}$ are algebraic over $\mathbb{F}_q$, so this is a finite extension of a finite field. Thus we have a sums-of-squares formula of type $[r,s,n]$ over a finite field.
\end{proof}

This corollary means that that sums-of-squares formulas can, in principle, be detected by checking over finite fields, and this can be done with a computer search. We do not know a priori how large of a field would be needed. The next section is dedicated to establishing a bound on the degree of this field. However, this bound is too large to be computationally feasible for conducting computer searches for formulas.

\section{Degree Bound}

In this section, we show that for fixed $[r,s,n]$ and prime $p$, there is $d$ such that if there is no sums-of-squares formula of type $[r,s,n]$ over $\mathbb{F}_{p^k}$ for $k<d$, then there is no sums-of-squares formula over $\bar{\mathbb{F}}_p$. Thus, we can theoretically check computationally if a sums-of-squares formula exists over $\bar{\mathbb{F}}$, since there are only finitely many choices for coefficients over $\mathbb{F}_{p^k}$. However, the value of $d$ is too large for this to be computationally feasible.

This result follows immediately from a theorem bounding the total degree of the zeta function due to Bombieri \cite{Bombieri}. We review the relevant material on zeta functions before giving our result. This material can be found in \cite{Wan}.

Let $\mathbb{F}_p$ be the field with $p$ elements, for $p$ prime. Let $X$ be an algebraic set over $\mathbb{F}_p$, we assume that $X$ is affine and defined by $m$ polynomials in $n$ variables $f_1,...,f_m\in\mathbb{F}_p[x_1,...,x_n]$, so 
\[
X(\mathbb{F}_p)=\{x\in\mathbb{F}_p^n\;|\;f_1(x)=\cdots=f_m(x)=0\}
\]
Let $d$ be the maximum total degree among the polynomials $f_i$.
 
Fix an algebraic closure $\bar{\mathbb{F}}_p$, and let $\mathbb{F}_{p^k}$ denote the unique subfield of $\bar{\mathbb{F}}_p$ with $p^k$ elements. $\#X(\mathbb{F}_{p^k})$ denotes the number of $\mathbb{F}_{p^k}$-rational points on $X$.

\begin{defn}
The \emph{zeta function} of $X$ is the generating function
\[
Z(X)=Z(X,T)=exp\left(\sum_{k=1}^\infty \frac{T^k}{k}\#X(\mathbb{F}_{p^k})\right)\in 1+T\mathbb{Z}[T]
\]
\end{defn}

As conjectured by Weil \cite{Weil} and proved by Dwork \cite{Dwork} and Grothendieck \cite{Grothendieck}, this is a rational function:

\begin{thm}
$Z(X)$ is a rational function, so it can be written
\[
Z(X,T)=\frac{R_1(X,T)}{R_2(X,T)}
\]
with $R_1,R_2\in\mathbb{Z}[T]$.

In fact, we can take $R_1,R_2\in 1+T\mathbb{Z}[T]$.
\end{thm}

The following theorem due to Bombieri \cite{Bombieri} gives an upper bound on the total degree $\mbox{deg }R_1+\mbox{deg }R_2$ of the zeta function.

\begin{thm}
With $R_1,R_2,d,n,m$ as above, we have:
\[
\mbox{deg }R_1+\mbox{deg }R_2<(4d+9)^{n+m}
\]
\end{thm}

The following algorithm, which can be found in \cite{Wan} shows that we can compute the zeta function by computing only finitely many of the $\#X(\mathbb{F}_{p^k})$ (specifically, for $k<2(4d+9)^{n+m}$). This determines the rest of the $\#X(\mathbb{F}_{p^k})$.

Let $D_1$ and $D_2$ be upper bounds for the degrees of $R_1$ and $R_2$, where $Z(X,T)=\frac{R_1(X,T)}{R_2(X,T)}$. For example, take $D_1=D_2=(4d+9)^{n+m}$ as given by the previous theorem.

Then, by counting $X(\mathbb{F}_{p^k})$ for $k\leq D_1+D_2$ (there are only finitely many elements in $\mathbb{F}_{p^k}^n$), we can compute the first $D_1+D_2+1$ terms in the power series $Z(X,T)$:
\[
Z(X,T)=1+z_1T+z_2T^2+\cdots+z_{D_1+D_2}T^{D_1+D_2}+\cdots
\]
Write
\begin{align*}
R_1(X)&=1+a_1T+\cdots +a_{D_1}T^{D_1}\\
R_2(X)&=1+b_1T+\cdots +b_{D_2}T^{D_2}
\end{align*}
Our goal is to determine the coefficients $a_i$ and $b_i$.

We have
\[
R_2(X)Z(X)\equiv R_1(X)\;(\mbox{mod }T^{D_1+D_2+1})
\]
This gives a linear system of equations in the $a_i$ and $b_i$. We know that it has a solution, which can be found using linear algebra. Thus the zeta function is determined.

This means that once we know $\#X(\mathbb{F}_{p^k})$ for $k\leq D_1+D_2$, the zeta function is determined, so $\#X(\mathbb{F}_{p^k})$ for $k>D_1+D_2$ are determined. In particular, if $\#X(\mathbb{F}_{p^k})=0$ for $k\leq D_1+D_2$, then $\#X(\mathbb{F}_{p^k})=0$ for all $k$, and $\#X(\bar{\mathbb{F}}_p)=0$. Applying this to the case of sums-of-squares formulas, we have the following result:

\begin{thm} \label{Ch4Bound}
Fix $[r,s,n]$ and prime $p$. If there is a sums-of-squares formula of type $[r,s,n]$ over $\bar{\mathbb{F}}_p$, then there is a sums-of-squares formula of type $[r,s,n]$ over $\mathbb{F}_{p^k}$ for some $k\leq 2\cdot 17^{rsn+r^2s^2}$.
\end{thm}

Thus, in principle, sums-of-squares formulas can be detected using searches over finite fields. However, the size of the fields that would be required make this unrealistic.

\section{Lower Bound on Independence}

In this section, we provide an explicit lower bound on $p$ for the characteristic $0$ and characteristic $p$ cases to be equivalent, and prove that existence of a sums-of-squares formula over an algebraically closed field is computable. We do this by analyzing the coefficients which appear throughout Buchberger's Algorithm. Buchberger's Algorithm produces Gr{\"o}bner bases, from which we can decide if an ideal is proper, thus determine if sums-of-squares formulas exist. Exposition can be found in \cite{Eisenbud}.

For the situation of sums-of-squares formulas, we find it most convenient to work with degree reverse lexicographic order, which we denote by $>$. By $\textrm{deg}$, we mean the total degree of the monomial.

For Buchberger's algorithm, we first require the Division algorithm. The following results can be found in \cite{Eisenbud}.

\begin{prop}
Let $S=k[x_1,...,x_r]$ be a polynomial ring over a field $k$ with a monomial order $>$. If $f,g_1,...,g_t\in S$, then there is an expression
\begin{center}
$f=\sum f_ig_i\,+f'$ with $f',f_i\in S$,
\end{center}
where none of the monomials of $f'$ is in $(\textrm{in}(g_1),...,\textrm{in}(g_t))$ and
\begin{center}
$\textrm{in}(f)\geq \textrm{in}(f_ig_i)$
\end{center}
for every $i$. $f'$ is called a \emph{remainder} of $f$ with respect to $g_1,...,g_t$. The expression $f=\sum f_ig_i\,+f'$ as above is called a \emph{standard expression} for $f$ in terms of the $g_i$.
\end{prop}

The Division Algorithm gives us an explicit construction of a standard expression.

\begin{prop} \textbf{(Division Algorithm)}
Let $S=k[x_1,...,x_r]$ be a polynomial ring over a field $k$ with a monomial order.If $f,g_1,...,g_t\in S$, then we can produce a standard expression
\[
f=\sum m_ug_{s_u}+f'
\]
for $f$ with respect $g_1,...,g_t$. We do this by defining the indices $s_u$ and terms $m_u$ inductively. Write
\[
f'_p=f-\sum_{u=1}^pm_ug_{s_u}\neq 0
\]
and $m$ is the maximal term of $f_p'$ that is divisible by some $\textrm{in}(g_i)$, then we take
\begin{align*}
s_{p+1}&=i\\
m_{p+1}&=m/\textrm{in}(g_i)
\end{align*}
The process ends when $f'_p=0$ or when no $\textrm{in}(g_i)$ divides a monomial of $f'_p$. The remainder $f'$ is the last $f'_p$ produced.
\end{prop}

This algorithm ends after finitely many steps since the maximal term of $f'_p$ divisible by some $\textrm{in}(g_i)$ decreases at each step.

Buchberger's Criterion allows us to determine if a set of elements is a Gr{\"o}bner basis.

\begin{prop} \textbf{(Buchberger's Criterion)} 
Let $S=k[x_1,...,x_r]$ be a polynomial ring with $k$ a field, and $g_1,...,g_t\in S$ nonzero elements.  For each pair of indices $i,j$, define
\[
m_{ij}=\textrm{in}(g_i)/GCD(\textrm{in}(g_i),\textrm{in}(g_j))\in S
\]
Then choose a standard expression
\[
m_{ji}g_i-m_{ij}g_j=\sum f_u^{(ij)}g_u+h_{ij}
\]
for $m_{ji}g_i-m_{ij}g_j$ with respect to $g_1,...,g_t$.

Then the elements $g_1,...,g_t$ form a Gr{\"o}bner basis if and only if $h_{ij}=0$ for all $i$ and $j$.
\end{prop}

This criterion leads to Buchberger's Algorithm, which allows us to construct Gr{\"o}bner bases.

\begin{prop} \textbf{(Buchberger's Algorithm)}
In the situation of Buchberger's Criterion, suppose $I\subset S$ is an ideal, and let $g_1,...,g_t\in I$ be a set of generators of $I$. If all the $h_{ij}=0$, then the $g_i$ form a Gr{\"o}bner basis for $I$. If some $h_{ij}\neq0$, then replace $g_1,...,g_t$ with $g_1,...,g_t,h_{ij}$, and repeat the process. The ideal generated by the initial terms of $g_1,...,g_t,h_{ij}$ is strictly larger than the ideal generated by the initial terms of $g_1,...,g_t$, so this process terminates after finitely many steps (since $S$ is noetherian). 
\end{prop}

Computing a Gr{\"o}bner basis would give us a clearer idea of the structure of the scheme of sums-of-squares formula. However, this is only computationally feasible for very small values of $r,s,n$. The main question we would like to answer is when $1$ is in the relevant ideal (since when the ideal is not proper, no sums-of-squares formula exists). By studying the coefficients that can arise in the Division Algorithm and Buchberger's Algorithm, we show that for ``large enough'' $p$, the algebraically closed cases in characteristic $0$ and characteristic $p$ coincide.

In order to establish this upper bound, we need some understanding of the complexity of Buchberger's Algorithm. This is provided by \cite{Dube}, in which it is shown:

\begin{thm}
Let $B$ be a set of polynomials with Gr{\"o}bner basis $G$. If there are $n$ variables and the polynomials in $B$ have total degree $\leq d$, then the degree of the polynomials in $G$ is bounded by
\[
2\left(\frac{1}{2}d^2+d\right)^{2^{n-1}}
\]
\end{thm}

In the case of sums-of-squares formulas of type $[r,s,n]$, we are considering a set $B$ of polynomials in $rsn$ variables with total degree $\leq 2$. Thus the degree of the polynomials in a Gr{\"o}bner basis is bounded by $2\cdot 4^{2^{rsn-1}}$.

We now use the above results to analyze the coefficients of the elements in a Gr{\"o}bner basis in the case of sums-of-squares formulas.

We have that a sums-of-squares formula of type $[r,s,n]$ exists over an algebraically closed field $F$ if and only if the ideal $I$ generated by
\[
\left\{\begin{array}{ll}
	\sum_i (x_{ijk}^2) -1 & \textrm{for all } j,k\\
	\sum_i x_{ijk}x_{ij'k} & \textrm{for } j < j' \textrm{ and all }k\\
	\sum_i x_{ijk}x_{ijk'} & \textrm{for all } j \textrm{ and } k < k'\\
	\sum_i (x_{ijk}x_{ij'k'}+x_{ijk'}x_{ij'k}) & \textrm{for } j<j' \textrm{ and } k<k'\\
\end{array}\right\}
\]
is proper in $S=F[\{x_{ijk}\}]$. This can be determined by computing a Gr{\"o}bner basis. Note that the polynomials in this generating set all have coefficients $\pm 1$ or $0$. We will carefully track the coefficients of polynomials over $\mathbb{Q}$ through the division algorithm and Buchberger's algorithm in order to obtain an upper bound on $p$ for which there can be a difference between the characteristic $0$ and characteristic $p$ cases.

For $\frac{a}{b}\in \mathbb{Q}$ with $a,b\in\mathbb{Z}$ relatively prime, we set the notation
\[
P\left(\frac{a}{b}\right)=\max\{\vert a\vert,\vert b\vert\}
\]
We prove some basic properties of $P$ before proceeding with the analysis of Buchberger's algorithm.

\begin{prop}
Let $x,y\in \mathbb{Q}$.
\begin{enumerate}
\item $P\left(\frac{1}{x}\right)=P(x)$
\item $P(-x)=P(x)$
\item $P(xy)\leq P(x)P(y)$
\item $P(x+y)\leq 2P(x)P(y)$
\end{enumerate}
\end{prop}

\begin{proof}
(1) and (2) are obvious.

Write $x=\frac{a}{b}$ with $a,b\in\mathbb{Z}$ relatively prime and $y=\frac{c}{d}$ with $c,d\in\mathbb{Z}$ relatively prime.
\begin{align*}
P(xy)&=P\left(\frac{ac}{bd}\right)\\
&=\max\{\vert ac\vert,\vert bd\vert\}\\
&\leq \max\{\vert a\vert,\vert b\vert\}\cdot\max\{\vert c\vert,\vert d\vert\}\\
&=P(x)P(y)\\
P(x+y)&=P\left(\frac{ad+bc}{bd}\right)\\
&=\max\{\vert ad+bc\vert,\vert bd\vert\}\\
&\leq 2\max\{\vert a\vert,\vert b\vert\}\cdot\max\{\vert c\vert,\vert d\vert\}\\
&=2P(x)P(y)
\end{align*}
\end{proof}

For simplicity, we will say that a polynomial $f$ has coefficients with $P$ bounded by $M$ if for every coefficient $a_i$ of $f$, we have $P(a_i)\leq M$.

\begin{prop}
Suppose we have $g_1,...,g_t\in S=\mathbb{Q}[x_1,...,x_r]$. Suppose further that every $g_i$ has total degree $\leq t$ and has coefficients with $P$ bounded by some number $M$. Then the $h_{ij}$ as computed in Buchberger's criterion have coefficients with $P$ bounded by 
\[
2^{3\cdot 2^p-2}M^{5\cdot 2^p-3}
\]
where $p=\left(\begin{array}{c}
rsn+2t\\
2t
\end{array}
\right)$.
\end{prop}

\begin{proof}
We have
\[
m_{ij}=\textrm{in}(g_i)/GCD(\textrm{in}(g_i),\textrm{in}(g_j))
\]
For convenience, we take the GCD to have coefficient $1$. Then, since $\textrm{in}(g_i)$ has coefficient with $P$ bounded by $M$, $m_{ij}$ has the same coefficient with $P$ bounded by $M$.

Then $m_{ji}g_i$ has coefficients with $P$ bounded by $M^2$, so $m_{ji}g_i-m_{ij}g_j$ has coefficients with $P$ bounded by $2M^2$. Note that the total degree of $m_{ji}g_i-m_{ij}g_j$ is $\leq 2t$.

We would like a bound on $h_{ij}$ in the standard expression
\[
m_{ji}g_i-m_{ij}g_j=\sum f_u^{(ij)}g_u+h_{ij}
\]
which we achieve in the following lemma, and the proposition follows immediately.
\end{proof}

\begin{lem} \label{GrobLemma}
Suppose we have $f,g_1,...,g_t\in S$, where $f$ has coefficients with $P$ bounded by $2M^2$ and every $g_i$ has coefficients with $P$ bounded by $M$. Suppose also that $f$ has total degree $d$. Then, in any standard expression
\[
f=\sum m_ug_{s_u}+f'
\]
the $f'$ has coefficients with $P$ bounded by
\[
2^{3\cdot 2^p-2}M^{5\cdot 2^p-3}
\]
where $p=\left(\begin{array}{c}
rsn+d\\
d
\end{array}
\right)$.
\end{lem}

\begin{proof}
For the first step of the division algorithm, we take $m$ to be the maximal term of $f$ that is divisible by some $\textrm{in}(g_i)$ and choose
\begin{align*}
s_1&=i\\
m_1&=m/\textrm{in}(g_i)
\end{align*}
$m$ has coefficient with $P$ bounded by $2M^2$ and $\textrm{in}(g_i)$ has coefficient with $P$ bounded by $M$, so $m_1$ has coefficient with $P$ bounded by $2M^3$.

Then
\[
f'_1=f-m_1g_{s_1}
\]
has coefficients with $P$ bounded by $2\cdot 2M^2\cdot2M^3M=8M^6$.

Suppose $f'_p$ has coefficients with $P$ bounded by $N$. Let $m$ be the maximal term of $f$ that is divisible by some $\textrm{in}(g_i)$, and choose
\begin{align*}
s_{p+1}&=i\\
m_{p+1}&=m/\textrm{in}(g_i)
\end{align*}
$m$ has coefficient with $P$ bounded by $N$ and $\textrm{in}(g_i)$ has coefficient with $P$ bounded by $M$, so $m_{p+1}$ has coefficient with $P$ bounded by $2M^2N$.

Then
\[
f'_{p+1}=f'_p-m_{p+1}g_{s_{p+1}}
\]
has coefficients with $P$ bounded by $2\cdot N\cdot 2M^2NM=4M^3N^2$.

Observe that the recursion $a_{n+1}=ca_n^2$ has closed form given by $a_n=c^{2^n-1}a_0^{2^n}$.

Thus, since $f'_0=f$ has coefficients with $P$ bounded by $2M^2$, we have that $f'_p$ has coefficients with $P$ bounded by $(4M^3)^{2^p-1}(2M^2)^{2^p}=2^{3\cdot 2^p-2}M^{5\cdot 2^p-3}$.

We now consider how many steps must be taken in the division algorithm. Now, $\textrm{in}(f'_p)<\textrm{in}(f'_{p-1})$ (in the monomial order) and $f$ has total degree $t$, so the total number of steps to produce the standard expression must be at most the number of monomials of total degree $\leq d$. (Note that this depends on us having chosen a homogeneous order.) In $rsn$ variables, this is given by:
\[
p=\left(
\begin{array}{c}
rsn+d\\
d
\end{array}
\right)
\]
\end{proof}

We now have a bound on $P$ in the $h_{ij}$ computed in Buchberger's Criterion. The next step is to consider how this develops as we proceed through Buchberger's Algorithm in the case of sums-of-squares formulas.

\begin{prop} \label{GrobProp}
Suppose $g_1,...,g_t$ generate an ideal $I$. Suppose also that the $g_i$ all have total degree $2$ and have coefficients with $P$ bounded by $1$. After $m$ steps of Buchberger's Algorithm, the candidates for a Gr{\"o}bner basis all have coefficients with $P$ bounded by
\[
\left(2^{3\cdot 2^p-2}\right)^{\left(\left(5\cdot 2^p-3\right)^{m+1}-1\right)/\left(5\cdot 2^p-4\right)}
\]
where $p=\left(\begin{array}{c}rsn+4^{2^{n-1}+1}\\4^{2^{n-1}+1}\end{array}\right)$.
\end{prop}

\begin{proof}
From Dub\'{e}'s result \cite{Dube}, we have that the degree of the polynomials in a Gr{\"o}bner basis for the ideal $I$ is bounded by $2\cdot 4^{2^{n-1}}$.

After $m$ steps of Buchberger's Algorithm, suppose the candidates for a Gr{\"o}bner basis all have coefficients with $P$ bounded by $a_m$. The previous proposition gives us the following recursion:
\[
a_m=2^{3\cdot 2^p-2}\left(a_{m-1}\right)^{5\cdot 2^p-3}
\]
where $p=\left(\begin{array}{c}rsn+2\left(2\cdot 4^{2^{n-1}}\right)\\2\left( 2\cdot 4^{2^{n-1}}\right)\end{array}\right)=\left(\begin{array}{c}rsn+4^{2^{n-1}+1}\\4^{2^{n-1}+1}\end{array}\right)$

A closed form of this recursion with $a_0=1$ is given by:
\[
a_m=\left(2^{3\cdot 2^p-2}\right)^{\left(\left(5\cdot 2^p-3\right)^{m+1}-1\right)/\left(5\cdot 2^p-4\right)}
\]
\end{proof}

Finally, by establishing an upper bound on the number of steps in Buchberger's Algorithm in our case, we arrive at our result.

\begin{thm} \label{BoundThm}
Fix $r,s,n$. Consider a prime $p$ with
\[\label{*}
p>\left(2^{3\cdot 2^q-2}\right)^{\left(\left(5\cdot 2^q-3\right)^{m+1}-1\right)/\left(5\cdot 2^q-4\right)}
\]
where $q=\left(\begin{array}{c}
rsn+4^{2^{n-1}+1}\\
4^{2^{n-1}+1}\end{array}\right)$ and $m=\left(\begin{array}{c}
rsn+2\cdot 4^{2^{n-1}}\\
2\cdot 4^{2^{n-1}}
\end{array}\right)$.

Then a sums-of-squares formula of type $[r,s,n]$ exists over an algebraically closed field of characteristic $0$ if and only if a sums-of-squares formula of type $[r,s,n]$ exists over an algebraically closed field of characteristic $p$.
\end{thm}

\begin{proof}
Once again, we use that the degree of polynomials in a Gr{\"o}bner basis is bounded by $2\cdot 4^{2^{n-1}}$. Note that, in each step Buchberger's Algorithm, the ideal generated by the initial terms of $g_1,...,g_t,h_{ij}$ is strictly larger than the ideal generated by the initial terms $g_1,...,g_t$. Since this is a monomial ideal, this means that the number of steps in Buchberger's algorithm is bounded by the number of monomials of degree $\leq 2\cdot 4^{2^{n-1}}$.

Thus the number of steps in Buchberger's Algorithm is bounded by:
\[
\left(\begin{array}{c}
rsn+2\cdot 4^{2^{n-1}}\\
2\cdot 4^{2^{n-1}}
\end{array}\right)
\]
Consider $I$ as above, the relevant ideal for sums-of-squares formulas. Every polynomial in a Gr{\"o}bner basis for $I$ has coefficients with $P$ bounded by \ref{*}. Then, for
\[p>\left(2^{3\cdot 2^q-2}\right)^{\left(\left(5\cdot 2^q-3\right)^{m+1}-1\right)/\left(5\cdot 2^q-4\right)}
\] 
(as in the statement of the theorem), we have that every coefficient of polynomials in the Gr{\"o}bner basis is nontrivial if and only if it is nontrivial when reduced modulo $p$. This is because the coefficients are all rational numbers, and neither the numerator nor denominator will become $0$ modulo $p$.

Recall that $I=S=F[\{x_{ijk}\}]$ if and only if no sums-of-squares formula of type $[r,s,n]$ exists over an algebraically closure of $F$. Considering Gr{\"o}bner bases, $I=S$ if and only if there is a non-zero constant in a Gr{\"o}bner basis, which we have shown is equivalent between the characteristic $0$ and characteristic $p$ cases for large enough $p$. Thus we have arrived at our result.
\end{proof}

Combining this result with those of the previous chapter, we have:

\begin{thm}
Existence of a sums-of-squares formula over an algebraically closed field is computable.
\end{thm}

\begin{proof}
For an algebraically closed field of finite characteristic $p$, we need only check the finitely many fields over $\mathbb{F}_p$ with degree given by \ref{Ch4Bound}. Thus there are finitely many possible coefficients, and this is computable.

For an algebraically closed field of characteristic $0$, we check existence over an algebraically closed field of finite characteristic $p$ for ``large enough'' $p$, as above.
\end{proof}

Unfortunately, the bounds are far to large for this to be practically computable.

\end{document}